\documentclass{amsart}

\usepackage{amsmath,amssymb,verbatim}
\usepackage{amsthm}
\usepackage{enumerate}
\usepackage{url}
\usepackage{mathtools}

\newtheorem{theorem}{Theorem}[section]
\newtheorem{lemma}[theorem]{Lemma}
\newtheorem{corollary}[theorem]{Corollary}
\newtheorem{proposition}[theorem]{Proposition}
\newtheorem{fact}[theorem]{Fact}

\theoremstyle{definition}
\newtheorem{definition}[theorem]{Definition}

\newtheorem{remark}[theorem]{Remark}

\def\seq{\subseteq}
\def\nv{\text{-}}
\def\inv{^{\text{-}1}}
\def\smd{\raisebox{.4pt}{\textrm{\scriptsize{~\!$\triangle$\!~}}}}
\def\cU{\mathcal{U}}
\def\cB{\mathcal{B}}
\def\cF{\mathcal{F}}
\def\cP{\mathcal{P}}
\def\cS{\mathcal{S}}
\def\N{\mathbb{N}}
\def\Z{\mathbb{Z}}
\def\Stab{\operatorname{Stab}}
\def\VC{\operatorname{VC}}
\def\Av{\operatorname{Av}}

\renewcommand{\phi}{\varphi}
\newcommand{\claim}{\hfill$\dashv_{\text{\scriptsize{claim}}}$}
\newcommand{\abar}{\bar{a}}

\AtBeginDocument{%
   \def\MR#1{}
}

\title[Quantitative structure of stable sets in groups]{Quantitative structure of stable sets in arbitrary finite groups}

\date{January 22, 2021}

\author[G. Conant]{Gabriel Conant}

\thanks{Partially supported by NSF grant DMS-1855503}

\address{Department of Pure Mathematics and Mathematical Statistics\\
University of Cambridge\\
Cambridge CB3 0WB\\
 UK}
\email{gconant@maths.cam.ac.uk}


\begin{document}

\begin{abstract}
We show that a $k$-stable set in a finite group can be approximated, up to given error $\epsilon>0$, by left cosets of a subgroup of index $\epsilon^{\nv O_k(1)}$. This  improves the bound in a similar result of Terry and Wolf on stable arithmetic regularity in finite abelian groups, and leads to a quantitative account of work of the author, Pillay, and Terry on stable sets in arbitrary finite groups. We also prove an analogous result for finite stable sets of small tripling in arbitrary groups, which provides a quantitative version of recent work by Martin-Pizarro, Palac\'{i}n, and Wolf. Our proofs use results on VC-dimension, and a finitization of model-theoretic techniques from stable group theory.
\end{abstract}

\maketitle

\section{Introduction}

In \cite{MaShStab}, Malliaris and Shelah established a surprising connection between graph regularity and dividing lines in first-order model theory. In particular, it had been known in the folklore that induced half-graphs witness the necessity for irregular pairs  in  Szemeredi's regularity lemma  (see \cite[Section 1.8]{KomSim}). Using techniques based on model-theoretic stability theory, Malliaris and Shelah proved the converse, i.e., if a finite graph omits half-graphs of a fixed size $k$ then, for any $\epsilon>0$, it admits an $\epsilon$-regular partition in the sense of Szemer\'{e}di, but with no irregular pairs. Moreover, the edge densities between regular pairs in the partition are within $\epsilon$ of $0$ or $1$, and the number of pieces in the partition is at most $\epsilon^{\nv O_k(1)}$ (versus an exponential tower of height $O(\epsilon^{\nv 2})$ in $\epsilon$-regular partitions of arbitrary graphs \cite{FoxLov,GowSRL}).\footnote{Throughout the paper, when we quantify over $\epsilon>0$, we tacitly assume $\epsilon\leq \frac{1}{2}$ to avoid inconsequential issues with asymptotic notation.}

 Since the work in \cite{MaShStab}, ``tame" graph regularity has been developed in several other model-theoretic settings (e.g., \cite{ChStNIP,ChStDis}). These results have further solidified the fruitful connection between model theory and combinatorics, e.g.,  via regularity  lemmas in the setting of bounded VC-dimension \cite{AFN, LovSzeg} (which predate \cite{MaShStab}). 

More recently, Terry and Wolf \cite{TeWo, TeWo2} developed a parallel connection between stability and arithmetic regularity in additive combinatorics, which was introduced by Green \cite{GreenSLAG} as a Fourier-analytic analogue of graph regularity for finite abelian groups. This led to an array of related work in the setting of stability and  bounded VC-dimension, including quantitative results of Alon, Fox, and Zhao \cite{AFZ} and Sisask \cite{SisNIP} for finite abelian groups, and qualitative results of the author, Pillay, and Terry \cite{CPT,CPTNIP} for arbitrary finite groups. 

Before stating precise results, we first define the relevant notion of stability in the setting of groups. As in the case of graphs, the definition is based on omitting half-graphs (or linear orders). Given a group $G$, we say that $A\seq G$ is \textbf{$k$-stable} if there do not exist $a_1,\ldots,a_k,b_1,\ldots,b_k\in G$ such that $a_ib_j\in A$ if and only if $i\leq j$.

\begin{theorem}[Terry \& Wolf \cite{TeWo2}]\label{thm:TW}
Suppose $G$ is a finite abelian group and $A\seq G$ is $k$-stable. Then for any $\epsilon>0$, there is a subgroup $H\leq G$ of index $\exp(\epsilon^{\nv O_k(1)})$ such that, for any $x\in G$, either $|(x+H)\cap A|<\epsilon|H|$ or $|(x+H)\backslash A|<\epsilon|H|$. So if $D$ is the union of all cosets $x+H$ such that $|(x+H)\cap A|\geq\epsilon|H|$, then $|A\smd D|<\epsilon|G|$.
\end{theorem}

Terry and Wolf first proved this result for $G=(\Z/p\Z)^n$, where $p$ is a fixed prime (see \cite{TeWo}). For comparison, Green's  arithmetic regularity lemma in $(\Z/2\Z)^n$ from \cite{GreenSLAG} says that for any $A\seq (\Z/2\Z)^n$ and $\epsilon>0$, there is a subgroup $H$ of index $m\leq \exp^{\epsilon^{\nv O(1)}}(1)$ such that $A$ is ``uniformly distributed"  in all but at most $\epsilon m$ cosets of $H$. Thus Theorem \ref{thm:TW} shows that stable subsets of finite abelian groups enjoy a version of arithmetic regularity with strengthened features analogous to those in  stable graph regularity.

Shortly after \cite{TeWo}, the author, Pillay, and Terry used model-theoretic techniques, in conjunction with an ultraproduct construction, to give a qualitative generalization of Theorem \ref{thm:TW} to arbitrary finite groups, but with ineffective bounds.

\begin{theorem}[C., Pillay, Terry \cite{CPT}]\label{thm:CPT}
Suppose $G$ is a finite group and $A\seq G$ is $k$-stable. Then for any $\epsilon>0$, there is a normal subgroup $H\leq G$ of index $O_{k,\epsilon}(1)$, and a set $D\seq G$ which is a union of  cosets of $H$, such that $|A\smd D|<\epsilon|H|$. Moreover, $H$ is a Boolean combination of bi-translates of $A$ of bounded complexity.
\end{theorem}

Note that Theorem \ref{thm:CPT} is qualitatively stronger than Theorem \ref{thm:TW}, since if $D$ is a union of cosets of $H$ then, for all $g\in G$, either $gH\cap A\seq A\smd D$ or $gH\backslash A\seq A\smd D$. In fact, the proof of Theorem \ref{thm:CPT} goes through a corresponding strengthening of the coset regularity behavior described in Theorem \ref{thm:TW} (see Remark \ref{rem:mainproof}$(2)$ for details). Together, Theorems \ref{thm:TW} and \ref{thm:CPT} raise the following natural questions, which appear (or are implied) in \cite{AFZ, CPT, CPTNIP, SisNIP, TeWo, TeWo2}. 

\begin{enumerate}
\item Can the exponential bound $\exp(\epsilon^{\nv O_k(1)})$ in Theorem \ref{thm:TW} be improved to a polynomial bound $\epsilon^{\nv O_k(1)}$ (to align with the case of graphs)?
\item Is there a comparably quantitative version of Theorem \ref{thm:TW} with the stronger structural approximation $|A\smd D|<\epsilon|H|$?
\item What is an effective bound for $O_{k,\epsilon}(1)$ in Theorem \ref{thm:CPT}?
\end{enumerate}

 The main goal of this article is a quantitative account of Theorem \ref{thm:CPT}, which answers these three questions. Toward this end, the following is our main result. 

\begin{theorem}[main result]\label{thm:SAR}
Suppose $G$ is a finite group and $A\seq G$ is $k$-stable. Then for any $\epsilon>0$, there is a subgroup $H\leq G$ of index $\epsilon^{\nv O_k(1)}$, and a set $D\seq G$ which is a union of left cosets of $H$, such that $|A\smd D|<\epsilon|H|$. 
\end{theorem}

Before addressing the differences between Theorems \ref{thm:SAR} and  \ref{thm:CPT}, we first compare Theorem \ref{thm:SAR} to  Theorem \ref{thm:TW}, and discuss the strategy of the proof. Note that Theorem \ref{thm:SAR} generalizes and strengthens Theorem \ref{thm:TW}, and provides positive answers to questions $(1)$ and $(2)$ above. The bound in Theorem \ref{thm:SAR} can also be written more explicitly as $O_k(\epsilon^{\nv N_k})$ where $N_k=k^{\exp^2(2k)}$ (see Remark \ref{rem:mainproof}$(1)$).

The proof methods used by Terry and Wolf in \cite{TeWo,TeWo2} combine discrete Fourier analysis in finite abelian groups together with an iterative construction based on the correspondence between coding orders and binary trees in graphs (see \cite[Theorem 2]{TeWo}). This correspondence was first proved by Shelah \cite{Shbook} via a set-theoretic argument based on a combinatorial theorem of Erd\H{o}s and Makkai \cite{ErdMak}, and was later made quantitatively explicit by Hodges \cite{Hodtrees}. It is also a key ingredient in Malliaris and Shelah's work on stable graph regularity. In \cite{TeWo2}, Terry and Wolf also employ tools involving regular Bohr sets and an almost-periodicity result of Sisask \cite{SisNIP} for sets of bounded VC-dimension in finite abelian groups.

Altogether, the connection between model theory and the work in  \cite{TeWo,TeWo2}  is largely grounded in pure stability theory and the notion of Shelah $2$-rank (i.e., Definition II.1.1 in \cite{Shbook} with $\lambda=2$). By contrast, our proof of Theorem \ref{thm:SAR} is more connected to  stable group theory and techniques involving measure-stabilizers and generic types. From a combinatorial perspective, these tools are similar to those used by Alon, Fox, and Zhao \cite{AFZ} in their arithmetic regularity lemma for sets of bounded VC-dimension in finite abelian groups of bounded exponent. 

To elaborate, suppose $G$ is a finite group and $A\seq G$ is $k$-stable. Given $\epsilon>0$, define the stabilizer $\Stab_\epsilon(A)=\{x\in G:|Ax\smd A|\leq\epsilon|G|\}$.  Using VC-theory, one can show that these  stabilizers are large in the sense that $G$ can be covered by $O_{k,\epsilon}(1)$ right translates of $\Stab_\epsilon(A)$ (see Corollary \ref{cor:VC}$(b)$). In Corollary \ref{cor:tech}, we show that for any  function $\sigma\colon (0,1)\to (0,1)$ and any $\epsilon>0$, there is some $\eta\geq\Omega_{\sigma,k,\epsilon}(1)$ such that $\Stab_\eta(A)\seq \Stab_{\sigma(\eta)}(A)$. By choosing $\sigma$ appropriately, we conclude  that $H:=\Stab_\eta(A)$ is a subgroup of $G$ and, moreover, we prove that any left coset of $H$ is either almost contained in $A$ or almost disjoint from $A$ up to small error in terms of $\epsilon$ and the index of $H$. This yields the statement of Theorem \ref{thm:SAR} modulo calculating polynomial bounds, which originate from VC-theory (in particular, \emph{Haussler's Packing Lemma}; see Theorem \ref{thm:VC} and Remark \ref{rem:VCdim}). 

The stabilizers defined above arise in additive combinatorics in the setting of ``popular difference sets" (see, e.g., \cite[Section 4]{LeLuSch}), and are also directly aligned with ingredients from stable group theory and its generalization in \cite{HPP} to  ``fsg" groups definable in NIP theories. In our work, the use stability is concentrated entirely in Corollary \ref{cor:tech} (described above).  Roughly speaking, this corollary reflects the result from \cite{CPT} that in a pseudofinite group, if $\cB$ is the Boolean algebra generated by the right translates of a fixed stable internal set, then the pseudofinite counting measure takes only finitely many values on $\cB$ (see Corollary \ref{cor:UP}). However, in \cite{CPT} this result follows as a \emph{consequence} of the model-theoretic tools developed for the proof of Theorem \ref{thm:CPT}, and relies on local stable group theory as formulated by Hrushovski and Pillay  in \cite{HrPiGLF}. Thus the focus of this paper is to obtain a more direct proof, which can be carried out quantitatively in the finite setting.

In order to fully recover Theorem \ref{thm:CPT}, and answer the third question above, we need to deal with normality of $H$ and ``definability" in terms of translates of $A$. This is done in Sections \ref{sec:norm} and \ref{sec:def}, where we prove the following additional results.

\begin{theorem}\label{thm:extras}
Suppose $G$ is a finite group and $A\seq G$ is $k$-stable. Fix $\epsilon>0$. 
\begin{enumerate}[$(a)$]
\item There is a subgroup $H\leq G$, which satisfies the conclusion of Theorem \ref{thm:SAR} and is of the form $\Stab_\eta(A)$ where $\epsilon^{O_k(1)}\leq\eta\leq\epsilon$. Moreover, there are $n\leq \epsilon^{\nv O_k(1)}$, $\ell\leq {n\choose \epsilon n}$, and $g_1,\ldots,g_n\in G$ such that $H=\bigcup_{t=1}^\ell \bigcap_{i=1}^n X_{t,i}$ where each $X_{t,i}$ is either $g_iA$ or $G\backslash g_iA$.
\item There is a normal subgroup $H\leq G$ satisfying the conclusion of Theorem \ref{thm:SAR}, except with index $\exp^{O_k(1)}(\epsilon\inv)$. Moreover, $H=\bigcap_{g\in G}gH_0g\inv$ for some subgroup $H_0\leq G$ satisfying the same description as in part $(a)$, but with $\eta\inv,n\leq \exp^{O_k(1)}(\epsilon\inv)$. 
\end{enumerate}
\end{theorem}

In part $(a)$, the expression of $H$ as a Boolean combination of right translates of $A$ is made possible by Lemma \ref{lem:tech} and the existence of $\epsilon$-approximations for sets systems of bounded VC-dimension (see Theorem \ref{thm:VCA}).
Note also that in part $(b)$, the number of conjugates of $H_0$ needed to obtain $H$ is bounded by the index of $H_0$. The iterated exponential bound arising in part $(b)$ is discussed after Theorem \ref{thm:SARn}. 

Recently in \cite{MPW}, Martin-Pizarro, Palac\'{i}n, and Wolf used model theoretic tools from local stability theory (including results from \cite{CPT} used for the proof of Theorem \ref{thm:CPT}), to obtain the following qualitative result for finite stable sets of small tripling in arbitrary groups.

\begin{theorem}[Martin-Pizarro, Palac\'{i}n, Wolf \cite{MPW}]\label{thm:MPW}
Suppose $G$ is a group and $A\seq G$ is a finite nonempty $k$-stable set with $|AAA|\leq c|A|$. Then for any $\epsilon>0$, there is a subgroup $H\leq G$ satisfying the following properties.
\begin{enumerate}[$(i)$]
\item $H\seq AA\inv$ and $A\seq CH$ for some $C\seq A$ of size $O_{k,c,\epsilon}(1)$.
\item There is a set $D\seq C$ such that $|A\smd DH|<\epsilon|A|$.
\end{enumerate}
\end{theorem}

The previous result is reminiscent of the \emph{Bogolyubov-Ruzsa Lemma} for abelian groups of bounded exponent (see \cite[Theorem 11.1]{SanBR}), which says that if $G$ is an abelian group of exponent $r$, and $A\seq G$ is a finite set with $|A+A|\leq c|A|$, then there is a subgroup $H\leq G$ contained in $2A-2A$ such that $A\seq C+H$ for some $C\seq A$ of size $\exp(O_r((\log 2c)^4))$. 
Part $(i)$ of Theorem \ref{thm:MPW} (with constant $\epsilon$) can be seen as a qualitative analogue of this result for for stable sets in arbitrary groups. In Section \ref{sec:MPW}, we prove a  version of Theorem \ref{thm:MPW}  with polynomial bounds.

\begin{theorem}\label{thm:SARtrip}
Suppose $G$ is a group and $A\seq G$ is a finite nonempty $k$-stable set with $|AA\inv A|\leq c|A|$. Then for any $\epsilon>0$, there is a subgroup $H\leq G$ satisfying the following properties.
\begin{enumerate}[$(i)$]
\item $H\seq A\inv A$ and $A\seq CH$ for some $C\seq A$ of size $O_k((c\epsilon\inv)^{O_k(1)})$.
\item There is a set $D\seq C$ such that $|A\smd DH|<\epsilon|H|$.
\end{enumerate}
\end{theorem}

The assumption on $A$ in the previous result is qualitatively weaker than in Theorem \ref{thm:MPW} since $|AAA|\leq c|A|$ implies $|AA\inv A|\leq c^4|A|$ (by \cite[Theorem 5.1]{RuzIrr}) but there is no uniform converse implication (see \cite[Remark 2.2]{CoBogo}). Also, condition $(ii)$ strengthens the corresponding part of Theorem \ref{thm:MPW} (after scaling $\epsilon$) since $|H|\leq c|A|$.

\section{Preliminaries}

We use $\log$ and $\exp$ for base $2$ logarithm and exponentiation. Given a function $f\colon X\to X$, and an integer $n\geq 1$, $f^n$ denotes the $n$-fold composition of $f$.

\subsection{Stable relations}

Given sets $X$ and $Y$, a \textbf{(binary) relation} on $X\times Y$ is a subset $\phi\seq X\times Y$. Following model-theoretic notation, we say ``$\phi(a,b)$ holds" for a given $a\in X$ and $b\in Y$ (and often just write $\phi(a,b)$) if $(a,b)\in\phi$. Given some fixed $b\in Y$, we let $\phi(X,b)$ denote the set $\{a\in X:\phi(a,b)\}$. We will frequently identify  the set $\phi(X,b)$ with the unary relation $\phi(x,b)$ on $X$.

\begin{definition}
A relation $\phi(x,y)$ on $X\times Y$ is \textbf{$k$-stable} if there do not exist $a_1,\ldots,a_k\in X$ and $b_1,\ldots,b_k\in Y$ such that $\phi(a_i,b_j)$ holds if and only if $i\leq j$. We say that $\phi(x,y)$ is \textbf{stable} if it is $k$-stable for some $k\geq 1$.
\end{definition}

Recall that a subset $A$ of a group $G$ is \textbf{$k$-stable} if and only if the binary relation ``$xy\in A$" is $k$-stable. It is a standard fact in model theory that stable relations are closed under Boolean combinations. In the setting of groups, we illustrate this with the following fact, which partly quotes Lemmas 1 and 2 of \cite{TeWo}. Given integers $k,\ell\geq 2$, let $R(k,\ell)$ denote the usual two-color Ramsey number for graphs.

\begin{fact}\label{fact:BA}
Let $G$ be a group, and let $\cB$ be the collection of stable subsets of $G$. Then $\cB$ is a bi-invariant Boolean algebra. In particular:
\begin{enumerate}[$(i)$]
\item If $A\seq G$ is $k$-stable then $Ag$ and $gA$ are $k$-stable for any $g\in G$.
\item If $A\seq G$ is $k$-stable then $G\backslash A$ is $(k+1)$-stable.
\item If $A\seq G$ is $k$-stable  and $B\seq G$ is $\ell$-stable, with $k,\ell\geq 2$, then $A\cap B$ is $R(k,\ell)$-stable and $A\cup B$ is $(R(k,\ell)+1)$-stable.
\end{enumerate}
Moreover, a subset of $G$ is $1$-stable if and only if it is empty, and a nonempty subset of $G$ is $2$-stable if and only if it is a coset of a subgroup of $G$.
\end{fact}
\begin{proof}
Parts $(i)$ and $(ii)$ are easy, and $(iii)$ follows from a routine Ramsey argument.  We leave the rest as an exercise (see also  \cite[Example 1]{TeWo} and \cite[Lemma 1.1]{SandersSR}). 
\end{proof}

\subsection{VC-dimension}\label{sec:VC}

\begin{definition}
Let $X$ be a set. A \textbf{set system on $X$} is a collection $\cS\seq\cP(X)$ of subsets of $X$. Given $A\seq X$, we say that $\cS$ \textbf{shatters} $A$ if $\cP(A)=\{A\cap S:S\in\cS\}$. The \textbf{VC-dimension} of $\cS$ is $\VC(\cS)=\sup\{n\in\N:\text{$\cS$ shatters some $A\in {X\choose n}$}\}$.
\end{definition}

We will use the following tools from VC-theory.

 \begin{theorem}\label{thm:VC}\textnormal{\cite{HaussPL,KoPaWo,VCdim}}
 Suppose $X$ is a finite set, and $\cS$ is a set system on $X$ with $\VC(\cS)\leq d$. Fix $\epsilon>0$.
 \begin{enumerate}[$(a)$]
 \item \textnormal{($\epsilon$-net Theorem)} There is $F\seq X$, with  $|F|\leq 8d\epsilon^{\nv 2}$, such that if $S\in\cS$ and $|S|>\epsilon|X|$, then $F\cap S\neq\emptyset$.
 \item \textnormal{(Haussler's Packing Lemma)} Suppose $\cF\seq\cS$ is such that $|U\smd V|>\epsilon|X|$ for all distinct $U,V\in\cF$. Then $|\cF|\leq (30\epsilon\inv)^d$.
 \end{enumerate}
 \end{theorem}
 
 \begin{remark}\label{rem:VCdim}
 The bound in part $(a)$ is not optimal, but will suffice for our purposes. The original work of Vapnik and Chervonenkis \cite{VCdim} yields $|F|\leq O(d\epsilon^{\nv2}\log(d\epsilon\inv))$ in part $(a)$ (see Theorem \ref{thm:VCA}).  Komlos, Pach, and Woeginger \cite{KoPaWo} improve this to $8d\epsilon\inv\log\epsilon\inv$.  Part $(b)$ of Theorem \ref{thm:VC} is due to Haussler \cite{HaussPL}. This statement can also be deduced fairly quickly from part $(a)$ and the Sauer-Shelah Lemma, but with the worse bound $(80d)^d\epsilon^{\nv 20d}$ (see \cite[Lemma 4.6]{LovSzeg}).
 \end{remark}
 
Next, we apply Theorem \ref{thm:VC} to various set systems defined in groups.

\begin{definition}
Let $G$ be a group and fix a set $A\seq G$.
\begin{enumerate}
\item Define $\VC_\ell(A)=\VC(\{xA:x\in G\})$ and $\VC_r(A)=\VC(\{Ax:x\in G\})$. 
\item Given a set $X\seq G$ and a real number $t>0$, we say $A$ is \textbf{(right) $t$-generic in $X$} if $X\seq AF$ for some $F\seq G$ with $|F|\leq t$.
\item Assume $A$ is finite. Given $\epsilon>0$ and finite set $X\seq G$, define 
\[
\Stab_\epsilon^X\!(A)=\{x\in G:|Ax \smd A|\leq\epsilon|X|\}.
\]
If $G$ is finite then we set $\Stab_\epsilon(A)=\Stab^G_\epsilon(A)$.
\end{enumerate}
\end{definition}

\begin{corollary}\label{cor:VC}
Let $G$ be a group, and fix finite sets $A,X\seq G$.
\begin{enumerate}[$(a)$]
\item If $\VC_\ell(A)\leq d$ and $|A|>\epsilon |X\inv A|$, then $A$ is $8d\epsilon^{\nv 2}$-generic in $X$.
\item If $\VC_r(A)\leq d$ and $|AX|\leq c|X|$, then $\Stab^{X}_\epsilon\!(A)$ is $(30c\epsilon\inv)^d$-generic in $X$.
\end{enumerate}
\end{corollary}
\begin{proof}
For part $(a)$, apply Theorem \ref{thm:VC}$(a)$ to the set system $\cS=\{gA:g\in X\inv\}$ on $X\inv A$, which has VC-dimension at most $\VC_\ell(A)$. 

Part $(b)$ is shown by Alon, Fox, and Zhao \cite[Lemma 2.2]{AFZ} in the case when $G$ is abelian and $G=X$. The general case is similar. Consider the set system $\cS=\{Ag:g\in X\}$ on $AX$, and note that $\VC(\cS)\leq\VC_r(A)$. Suppose $\VC_r(A)\leq d$ and let $S=\Stab^{X}_\epsilon(A)$. Choose $F\seq X$ of maximal size such that $|Ag\smd Ah|>\epsilon|X|$ for all distinct $g,h\in F$. Then $|F|\leq (30c\epsilon\inv)^d$ by Theorem \ref{thm:VC}$(b)$. For any $g\in G$ there is some $h\in F$ such that $|Ag\smd Ah|\leq\epsilon|X|$, i.e., $gh\inv\in S$. So $X\seq SF$.
\end{proof}

Finally, we relate VC-dimension to stability. Suppose $\phi(x,y)$ is a relation on $X\times Y$. It is easy to check that if $\phi(x,y)$ is $k$-stable then the VC-dimension of the set system $\{\phi(X,b):b\in Y\}$ on $X$ is strictly less than $k$. Thus:

\begin{fact}\label{fact:VCstab}
If $G$ is a group and $A\seq G$ is $k$-stable, then $\VC_\ell(A)<k$ and $\VC_r(A)<k$.
\end{fact}

\section{Proof of Theorem \ref{thm:SAR}}\label{sec:proof}

In the proof of Theorem \ref{thm:SAR}, the stability assumption will be leveraged against an auxiliary relation  used in the definition of stabilizers. Before examining this relation, we first prove the main technical lemma in a  more general setting.

\begin{definition}
Let $G$ be a group and suppose $\phi(x,y)$ is a relation on $G\times Y$ for some set $Y$. Then $\phi(x,y)$ is \textbf{right-invariant} if, for any $b\in Y$ and $g\in G$, there is $c\in Y$ such that $\phi(G,b)g=\phi(G,c)$.
\end{definition}

\begin{remark}\label{rem:stabinv}
If $\phi(x,y)$ is a $k$-stable right-invariant relation on $G\times Y$ then, for any $b\in Y$, $\phi(G;b)$ is a $k$-stable subset of $G$.
\end{remark}

We now prove the main technical lemma in the paper.  This result also represents the only direct use of stability in the proof of Theorem \ref{thm:SAR} (see Remark \ref{rem:mainproof}$(3)$).

\begin{definition}
Given a function $\sigma\colon (0,1)\to (0,1)$, an integer $k\geq 2$, and a real number $r\geq 1$, define $\sigma_{k,r}\colon (0,1)\to (0,1)$ so that $\sigma_{k,r}(x)=x\sigma(\frac{1}{2}x)^2/(8(k-1)r^2)$.
\end{definition}

\begin{lemma}\label{lem:tech}
Let $G$ be a group and suppose $\phi(x,y)$ is a $k$-stable right-invariant relation on $G\times Y$ for some set $Y$ and some $k\geq 2$. Fix a  nonempty finite set $X\seq G$ and suppose $|X\inv X|\leq r|X|$. Then, for any increasing function $\sigma\colon (0,1)\to (0,1)$ and any $\epsilon>0$, there is some $\eta\in [(\sigma_{k,r})^k(\epsilon),\epsilon]$ such that, for any $b\in Y$, if $\phi(G,b)\seq X$ and $|\phi(G,b)|\leq \eta|X|$ then $|\phi(G,b)|\leq \sigma(\eta)|X|$.
\end{lemma}

Before starting the proof, we remark that Theorem \ref{thm:SAR} will only require the special case where $G$ is finite and $X=G$ (so  we may take $r=1$).

\begin{proof}
Suppose the result fails for some function $\sigma$ and some $\epsilon>0$. Set $\tau=\sigma_{k,r}$ and $\delta=\tau^k(\epsilon)$.  Let $Y_X=\{b\in Y:\phi(G,b)\seq X\}$.  Given a finite set $A\seq G$, let $\mu(A)=|A|/|X|$. Then we have:
\begin{equation*}
\text{For any $\eta\in [\delta,\epsilon]$, there is some $b\in Y_X$ such that $\sigma(\eta)<\mu(\phi(x;b))\leq\eta$.}\tag{$\dagger$}
\end{equation*}

Given $n\geq 1$ and $1\leq t\leq n$, define the following relation on $G\times Y^n$:
\[
\textstyle\phi^n_t(x;y_1,\ldots,y_n)=\bigwedge_{i=1}^t\phi(x,y_i)\wedge\bigwedge_{i=t+1}^n\neg\phi(x,y_i).
\]
By induction on $1\leq n\leq k$, we will construct $b_1,\ldots,b_n\in Y$ such that $b_1\in Y_X$ and, for all $1\leq t\leq n$, $\mu(\phi^n_t(x;b_1,\ldots,b_t))> \tau^n(\epsilon)$.  Given this, we can then choose $a_1,\ldots,a_{k}\in G$ such that $\phi^{k}_t(a_t;b_1,\ldots,b_{k})$ holds for all $1\leq t\leq k$, which yields $\phi(a_i,b_j)$ if and only if $i\geq j$. Setting $a^*_i=a_{k-i+1}$ and $b^*_j=b_{k-j+1}$, one then obtains $\phi(a^*_i,b^*_j)$ if and only if $i\leq j$, contradicting $k$-stability of $\phi(x,y)$. 

For the base case $n=1$, apply $(\dagger)$ with $\eta=\epsilon$ to find $b_1\in Y_X$ such that $\mu(\phi(x,b_1))> \sigma(\epsilon)$. Since $\sigma(\epsilon)>\tau(\epsilon)$ and $\phi^1_1=\phi$, we are done.  

Now fix $1\leq n<k$ and suppose we have $b_1,\ldots,b_n\in Y$ satisfying the desired conditions. Set $d=k-1$ and $\eta=\frac{1}{2}\tau^n(\epsilon)$, and note that $\eta\in[\delta,\epsilon]$. By $(\dagger)$, there is some $c\in Y_X$ such that $\sigma(\eta)<\mu(\phi(x,c))\leq\eta$. Let  $C:=\phi(G,c)$. Then $C$ is $k$-stable as a subset of $G$ (by Remark \ref{rem:stabinv}) and $|C|>\sigma(\eta)|X|\geq r\inv\sigma(\eta)|X\inv C|$. So  $C$ is $8dr^2\sigma(\eta)^{\nv 2}$-generic in $X$ by Corollary \ref{cor:VC}$(a)$ and Fact \ref{fact:VCstab}. Let $B=\phi^n_n(G;b_1,\ldots,b_n)$. Then $B\seq X$ and $|B|> \tau^n(\epsilon)|X|$ by induction. By averaging, we can find some $g\in G$ such that $|B\cap Cg|> (\tau^n(\epsilon)\sigma(\eta)^2/8dr^2)|X|$. 

Since $\phi$ is right-invariant, there is some $b_{n+1}\in Y$ such that $\phi(G,b_{n+1})=Cg$. So we have
\begin{multline*}
\mu(\phi^{n+1}_{n+1}(x;b_1,\ldots,b_{n+1}))=\mu(\phi^n_n(x;b_1,\ldots,b_n)\wedge\phi(x,b_{n+1}))\\
> \tau^n(\epsilon)\sigma(\eta)^2/8dr^2=\tau^{n+1}(\epsilon). 
\end{multline*} 
Now fix $1\leq t\leq n$. Then
\[
\phi^{n+1}_t(x;b_1,\ldots,b_{n+1})=\phi^n_t(x;b_1,\ldots,b_n)\wedge\neg\phi(x,b_{n+1}).
\]
Since $\mu(\phi^n_t(x;b_1,\ldots,b_n))> \tau^n(\epsilon)$ by induction, and
\[
\textstyle\mu(\phi(x,b_{n+1}))=\mu(\phi(x,c))\leq \eta=\frac{1}{2}\tau^n(\epsilon),
\]
it follows that $\mu(\phi^{n+1}_t(x;b_1,\ldots,b_{n+1}))>\frac{1}{2}\tau^n(\epsilon)>\tau^{n+1}(\epsilon)$. Altogether, $b_1,\ldots,b_{n+1}$ satisfy the desired properties.
\end{proof}

\begin{definition}
Given a group $G$ and a set $A\seq G$, let $\phi_A(x;y,z)$ be the relation on $G\times G^2$ defined by $x\in Ay\smd Az$.
\end{definition}

\begin{corollary}\label{cor:tech}
Let $G$ be a finite group and suppose $A\seq G$ is such that $\phi_A(x;y,z)$ is $k$-stable for some $k\geq 2$. Then, for any increasing function $\sigma\colon (0,1)\to (0,1)$ and any $\epsilon>0$, there is some $\eta\in [(\sigma_{k,1})^k(\epsilon),\epsilon]$ such that $\Stab_\eta(A)\seq \Stab_{\sigma(\eta)}(A)$.
\end{corollary}
\begin{proof}
Note that $\phi_A(x;y,z)$ is right-invariant. So we can apply Lemma \ref{lem:tech} with $X=G$, which yields the desired result. 
\end{proof}

\begin{proposition}\label{prop:smd}
Suppose $G$ is a group and $A\seq G$ is $k$-stable, with $k\geq 2$. Then $\phi_A(x;y,z)$ is $n$-stable, where $n=R\big(R(k,k+1),R(k,k+1)\big)+1$.
\end{proposition}
\begin{proof}
If $\psi(x,y)$ denotes the relation $x\in Ay$ on $G\times G$, then $\psi(x,y)$ is $k$-stable and $\phi_A(x;y,z)$ is equivalent to $(\psi(x,y)\wedge\neg\psi(x,z))\vee(\psi(x,z)\wedge\neg\psi(x,y))$. So the result follows from the same Ramsey argument underlying Fact \ref{fact:BA}. 
\end{proof}

By the previous proposition, we can make the following definition. 

\begin{definition}
Given an integer $k\geq 2$, define $k_*$ to be the least integer $n\geq 1$ such that $\phi_A(x;y,z)$ is $n$-stable for any group $G$ and any $k$-stable set $A\seq G$.
\end{definition}

We now prove the main result.

\begin{proof}[\textnormal{\textbf{Proof of Theorem \ref{thm:SAR}}}]
As noted in Fact \ref{fact:BA}, the situation is trivial for $k=1$. So
fix $k\geq 2$ and $\epsilon>0$. Without loss of generality, assume $\epsilon< (8(k-1)30^{4k-4})\inv$. Let $\sigma(x)=x^{4k}$ and set $\delta=(\sigma_{k_*,1})^{k_*}(\epsilon)$. Then $\delta\geq \epsilon^{O_k(1)}$ (see also Remark \ref{rem:mainproof}$(1)$).

 Now fix a finite group $G$ and a $k$-stable set $A\seq G$. By Corollary \ref{cor:tech}, there is some $\eta\in[\delta,\epsilon]$ such that $\Stab_\eta(A)=\Stab_{\eta^{4k}}(A)$. Note  that $2\eta^{4k}\leq\eta$. Therefore, if we set $H=\Stab_\eta(A)$, then $H=H\inv$ and $HH\seq \Stab_{2\eta^{4k}}(A)\seq H$. So $H$ is a subgroup of $G$. By Corollary \ref{cor:VC}$(b)$, $H$ has index $m$ where 
 \[
 m\leq (30\eta\inv)^{k-1}\leq (30\delta\inv)^{k-1}\leq \epsilon^{\nv O_k(1)}.
 \]

\noindent\emph{Claim:} If $g\in G$ then either $|gH\cap A|<m\inv\eta|H|$ or $|gH\backslash A|<m\inv\eta|H|$.

\noindent\emph{Proof:} 
Fix $g\in G$ and, toward a contradiction, suppose we have $|gH\cap A|\geq m\inv\eta|H|$ and $|gH\backslash A|\geq m\inv\eta|H|$. Let $B=H\cap g\inv A$ and $C=H\backslash g\inv A$. Then $B,C\seq H$ and $|B|,|C|\geq m\inv\eta|H|$. Note that $B$ is $k$-stable by Fact \ref{fact:BA} (and since $R(k,2)=k$). So by Corollary \ref{cor:VC}$(a)$ and Fact \ref{fact:VCstab}, $B$ is $t$-generic in $H$, where $t= 8(k-1)(m\eta\inv)^2$. 
By averaging, we can find some $h\in G$ such that 
 \[
 |Bh\cap C|\geq \frac{1}{t}|C|\geq  \frac{\eta}{mt}|H|=\frac{\eta}{m^2t}|G|\geq \frac{\eta^{4k-1}}{8(k-1)30^{4k-4}}|G|>\eta^{4k}|G|.
 \] 
 Note that $Bh\cap C=H\cap Hh\cap g\inv (Ah\backslash A)$, and so $h\in H$ and $|Ah\backslash A|>\eta^{4k}|G|$. But this is a contradiction, since $H=\Stab_{\eta^{4k}}(A)$ and $Ah\backslash A\seq Ah\smd A$. \claim
\medskip

Now let $C_1,\ldots,C_m$ enumerate the left cosets of $H$ in $G$. Let $I$ be the set of $i\leq m$ such that $|C_i\cap A|\geq m\inv\eta|H|$. Set $D=\bigcup_{i\in I}C_i$. Then $D\backslash A\seq \bigcup_{i\in I}C_i\backslash A$ and $A\backslash D\seq \bigcup_{i\not\in I} C_i\cap A$.  Note that if $i\in I$ then $|C_i\backslash A|<m\inv\eta|H|$ by the claim; and if $i\not\in I$ then $|C_i\cap A|<m\inv\eta|H|$. So we have 
\[
|A\smd D|\leq \sum_{i\in I}|C_i\backslash A|+\sum_{i\not\in I}|C_i\cap A|< \sum_{i=1}^mm\inv\eta|H|=\eta|H|\leq\epsilon|H|.\qedhere
\]
\end{proof}

\begin{remark}\label{rem:mainproof}
We make some comments on the proof of Theorem \ref{thm:SAR}.
\begin{enumerate}[$(1)$]
\item If $\sigma(x)=x^{4k}$ and $\epsilon>0$, then $(\sigma_{k_*,1})^{k_*}(\epsilon)= c\epsilon^{(8k+1)^{k_*}}$ for some $c=c(k)$. So by Proposition \ref{prop:smd}, and the rough bound $R(k,\ell)\leq\exp(k+\ell-2)$, we have $m\leq O_k(\epsilon^{\nv N_k})$ where $N_k=k^{\exp^2(2k)}$.
\item The set $D$ approximates $A$ up to an ``error set" of size at most $\epsilon|H|$. But the proof further shows this error set is evenly distributed in the cosets of $H$. In particular, for all $g\in G$, we have $|gH\cap A|<m\inv \epsilon|H|$ or $|gH\backslash A|<m\inv \epsilon|H|$.

\item The proof only requires a fixed bound on $\VC_\ell(A)$, $\VC_r(A)$, and the stability of $\phi_A(x;y,z)$. Although we do not have an example on hand to discern whether this is weaker than stability of $A$ as a set, it can easily be seen that some level of stability is needed for Theorem \ref{thm:SAR}. For example, let $G=\Z/p\Z$ and let $A=\{0,1,\ldots,\frac{p-1}{2}\}$, where $p>2$ is prime. Then $\VC_\ell(A)=2$, but $A$ cannot be approximated by a subgroup of $G$ whose index is independent of $p$. 
\end{enumerate}
\end{remark}

\section{Normality, definability, and tripling}

\subsection{Normality}\label{sec:norm}

In contrast to Theorem \ref{thm:CPT}, we do not necessarily know that the subgroup given by Theorem \ref{thm:SAR} is \emph{normal}. While normality is naturally motivated as a group-theoretic property, it is also crucial when using cosets to obtain a regular partition of the ``Cayley product graph" associated to stable subsets of finite groups (see \cite[Corollary 3.5]{CPT}).
 In order to obtain a normal subgroup in this situation, we must introduce an exponential tower of bounded height.

\begin{theorem}\label{thm:SARn}
Suppose $G$ is a finite group and $A\seq G$ is $k$-stable. Then for any $\epsilon>0$, there is a normal subgroup $H\leq G$ of index $\exp^{O_k(1)}(\epsilon\inv)$, and a set $D\seq G$ which is a union of cosets of $H$, such that $|A\smd D|<\epsilon|H|$.
\end{theorem}
\begin{proof}
We just explain how to modify the proof of Theorem \ref{thm:SAR} (from Section \ref{sec:proof}). Fix $k\geq 2$ and $\epsilon>0$, and assume $\epsilon$ is sufficiently small (depending only on $k$) so that if $\eta\leq\epsilon$ then $(\dagger)$ below is satisfied.

 Let $\sigma(x)=\exp(\nv x^{\nv k})$ and set $\delta=(\sigma_{k_*,1})^{k_*}(\epsilon)$.  Apply Corollary \ref{cor:tech} to obtain $\eta\in[\delta,\epsilon]$ with $H_0=\Stab_\eta(A)=\Stab_{\sigma(\eta)}(A)$.  As before, $H_0$ is a subgroup of $G$ of index $m_0\leq (30\eta\inv)^{k-1}$. So $H:=\bigcap_{g\in G}gH_0g\inv$ is a \emph{normal} subgroup of $G$ of index  $m\leq m_0!\leq \exp(m_0\log m_0)\leq\exp^{O_k(1)}(\epsilon\inv)$.
Now we follow the same claim as in the proof of Theorem \ref{thm:SAR}. In particular, we obtain $B,C\seq H$ and $h\in H$ such that
\[
|Bh\cap C|>\frac{\eta}{m^2t}|G|\geq \frac{\eta^3}{8(k-1)\exp[4(30\eta\inv)^{k-1}\log((30\eta\inv)^{k-1})]}|G|,
\]
where $t=8(k-1)(m\eta\inv)^2$. So, assuming
\begin{equation*}
\exp[\eta^{\nv k}-4(30\eta\inv)^{k-1}\log((30\eta\inv)^{k-1})]>8(k-1)\eta^{\nv 3},\tag{$\dagger$}
\]
we have $Bh\cap C>\sigma(\eta)|G|$. Since $h\in H\leq\Stab_{\sigma(\eta)}(A)$, this leads to a similar contradiction. 
\end{proof}

The iterated exponential bound in the previous theorem is of course much worse than the polynomial bound  in Theorem \ref{thm:SAR}. However, viewing $k$ as fixed, it is still much better than the bound in the general arithmetic regularity lemma for finite abelian groups from \cite{GreenSLAG}, which lies between $\exp^{\Omega(\epsilon\inv)}(1)$ and $\exp^{\epsilon^{\nv 3}}(1)$ (see \cite{HLMS}). In light of Theorem \ref{thm:TW}, we could instead ask for a normal subgroup $H\leq G$, with a better bound on the index, but such that only the ``regularity" aspect holds, i.e., for all $g\in G$, either $|gH\cap A|<\epsilon|H|$ or $|gH\backslash A|<\epsilon|H|$. As previously noted, this leads to the weaker approximation $|A\smd D|<\epsilon|G|$ where $D\seq G$ is a union of cosets of $H$. Such a statement can nearly be obtained by combining Lemma \ref{lem:tech} with techniques from \cite{AFZ} (and their extensions in \cite{CoBogo}).

\begin{proposition}\label{prop:error}
Suppose $G$ is a finite group and $A\seq G$ is $k$-stable. Then for any $\epsilon>0$, there is a normal subgroup $H\leq G$ of index $\exp(\epsilon^{\nv O_k(1)})$, and a set $Z\seq G$ which is a union of cosets of $H$ with $|Z|<\epsilon|G|$, such that for any $g\in G\backslash Z$, either $|gH\cap A|< \epsilon|H|$ or $|gH\backslash A |< \epsilon|H|$. Moreover, there is a set $D\seq G$, which is a union of cosets of $H$, such that $|A\smd D|< \epsilon|G|$.
\end{proposition}
\begin{proof}
Fix $\epsilon>0$ and assume $k\geq 2$. Apply Corollary \ref{cor:tech} with $\epsilon^4$ and $\sigma(x)=\frac{1}{2}x$ to obtain a subgroup $H_0:=\Stab_\eta(A)=\Stab_{\eta/2}(A)$ with $\epsilon^{O_k(1)}\leq \eta\leq\epsilon^4$. Now let $H=\bigcap_{g\in G}gH_0g\inv$. So $H$ is normal in $G$ and, by Corollary \ref{cor:VC}$(b)$, has index $\exp(\epsilon^{\nv O_k(1)})$. Since $H\seq\Stab_{\eta/2}(A)$, we can apply \cite[Lemma 8.2]{CoBogo}\footnote{This result generalizes and elaborates on methods in the proof of \cite[Lemma 2.4]{AFZ}, which focuses on abelian groups.} to obtain $D,Z\seq G$, which are each unions of cosets of $H$, such that $|A\smd D|<\eta|G|$, $|Z|<\frac{1}{2}\eta^{1/2}|G|$, and, for any $g\in G\backslash Z$, either $|gH\cap A|<\eta^{1/4}|H|$ or $|gH\backslash A|<\eta^{1/4}|H|$. Since $\eta\leq\epsilon^4$, we have the desired conclusions.
\end{proof}

Despite the improved bound in Proposition \ref{prop:error}, the conclusion is qualitatively sub-optimal due to the error set $Z$. Indeed, the absence of ``irregular cosets"  is one of the hallmarks of stability in the setting of arithmetic regularity. For example, by Theorem 1.6 and Remark 8.3 in \cite{CoBogo}, Proposition \ref{prop:error} holds under the weaker assumption that $\VC_\ell(A)<k$, but with the bound $\exp(O_{k,r}(\epsilon^{\nv k}))$ where $r$ is the exponent of $G$ (the dependence on $r$ is necessary by the example in Remark \ref{rem:mainproof}$(3)$).

\subsection{Definability}\label{sec:def}

In Theorem \ref{thm:SAR}, the subgroup $H$ is built from the set $A$, in the sense that $H=\Stab_\eta(A)$ where $\epsilon^{O_k(1)}\leq\eta\leq\epsilon$. On the other hand, Theorem \ref{thm:CPT} provides  ``first-order, quantifier-free" definability of $H$ from bi-translates of $A$. In this section, we use $\epsilon$-approximations for  set systems of bounded VC-dimension to recover this feature quantitatively. 

\begin{definition}
Given a set $X$ and a tuple $\abar=(a_1,\ldots,a_n)\in X^n$,  define the map  $\Av_{\abar}\colon\cP(X)\to[0,1]$ so that $\Av_{\abar}(S)=\frac{1}{n}|\{1\leq i\leq n: a_i\in S\}|$.
\end{definition}

\begin{theorem}[Vapnik \& Chervonenkis \cite{VCdim}]\label{thm:VCA}
Suppose $X$ is a nonempty finite set and $\cS\seq\cP(X)$ satisfies $\VC(\cS)\leq d$. Then for any $\epsilon>0$, there is a tuple $\abar\in X^n$, with $n\leq O(d\epsilon^{\nv 2}\log(d\epsilon\inv))$, such that for any $S\in \cS$, $||S|/|X|-\Av_{\abar}(S)|\leq\epsilon$.
\end{theorem}

\begin{proposition}\label{prop:bool}
Suppose $G$ is a finite group and $A\seq G$ satisfies $\VC_r(A)\leq d$. Assume that $\Stab_\kappa(A)=\Stab_\lambda(A)$ for some $0<\lambda<\kappa<1$. Then there are $g_1,\ldots,g_n\in G$, for some $n\leq O(d(\kappa-\lambda)^{\nv 2}\log(d(\kappa-\lambda)\inv))$,  such that 
\[
\textstyle\Stab_\kappa(A)=\bigcup_{t=1}^{\ell} \bigcap_{i=1}^n X_{t,i},
\]
where $\ell\leq {n\choose (\kappa+\lambda)n/2}$ and each $X_{t,i}$ is either $g_iA$ or $G\backslash g_iA$.
\end{proposition}
\begin{proof}
Let $\cS=\{Ax\smd A:x\in G\}$. Then $\VC(\cS)\leq 10d$ by \cite[Lemma 4.4]{LovSzeg}.
Let $\theta=(\kappa-\lambda)/2$. Given $x\in G$, set $\mu_x=|Ax\smd A|/|G|$. By Theorem \ref{thm:VCA}, there is a tuple $\abar=(a_1,\ldots,a_n)\in G^n$, for some $n\leq O(d\theta^{\nv 2}\log(d\theta\inv))$,  such that for all $x\in G$, if we set $\alpha_x=\Av_{\abar}(Ax\smd A)$, then $|\mu_x-\alpha_x|\leq\theta$.

Now define $B=\left\{x\in G:\alpha_{x}\leq\lambda+\theta\right\}$. If $x\in B$, then $\mu_x\leq\alpha_x+\theta\leq\lambda+2\theta=\kappa$, and so $x\in \Stab_\kappa(A)$. Conversely, if $x\in \Stab_\kappa(A)$ then $x\in \Stab_\lambda(A)$, and so $\alpha_x\leq\mu_x+\theta\leq\lambda+\theta$, i.e.,  $x\in B$. So $B=\Stab_\kappa(A)$. 

Finally, we show that $B$ is a Boolean combination of the desired form. Given $a\in G$, define $Z_a=\{x\in G:a\in Ax\inv \smd A\}$. Then $Z_a=a\inv A$ if $a\not\in A$, and $Z_a=G\backslash a\inv A$ if $a\in A$. Given $\sigma\seq[n]$, define
\[
\textstyle Y_\sigma=\bigcap_{i\in \sigma}Z_{a_i}\cap\bigcap_{i\not\in\sigma}G\backslash Z_{a_i}.
\]
Let $\Sigma=\{\sigma\seq [n]:|\sigma|/n\leq \lambda+\theta\}$. Then it follows from the definitions, and symmetry of $B$, that $B=\bigcup_{\sigma\in\Sigma}Y_\sigma$. So $B=\Stab_\kappa(A)$ has the desired form.
\end{proof}

\begin{remark}
The previous proposition can also be seen as a special case of ``definability" of probability measures that are finitely approximable in the sense of Theorem \ref{thm:VCA}. This situation is dealt with in much greater generality in \cite{ChStNIP}. 
\end{remark}

\begin{proof}[\textnormal{\textbf{Proof of Theorem \ref{thm:extras}}}]
For part $(a)$, note that by the proof of Theorem \ref{thm:SAR}, we have $H=\Stab_\eta(A)=\Stab_{\sigma(\eta)}(A)$ where $\epsilon^{O_k(1)}\leq\eta\leq\epsilon$ and $\sigma(\eta)\leq \frac{1}{2}\eta$. So apply Proposition \ref{prop:bool} with $\kappa=\eta$ and $\lambda=\sigma(\eta)$. Part $(b)$ follows similarly, using Theorem \ref{thm:SARn} (and its proof).
\end{proof}

\subsection{Tripling}\label{sec:MPW}

In this section, we prove Theorem \ref{thm:SARtrip}, which gives a structural approximation of finite stable sets satisfying a weak form of small tripling (which is called ``small alternation" in \cite{CoBogo}). To motivate this assumption, consider a group $G$ and a finite nonempty set $A\seq G$. Then $\Stab^A_\epsilon(A)\seq A\inv A$ for any $\epsilon>0$ since, if $|Ax\smd A|<\epsilon|A|$ then $Ax\cap A\neq\emptyset$, i.e., $x\in A\inv A$. So while $G$ may be infinite, we can use iterated products of $A$ and $A\inv$ as ``finite domains" for $\Stab^A_\epsilon(A)$ and various translates. In order for this to be useful, we need to bound the size of such products in terms of $|A|$, and this is where small alternation comes into play. Specifically, we will use the following consequence of the triangle inequality for \emph{Ruzsa distance}  (see, e.g., the proof of \cite[Proposition 3.2$(b)$]{CoBogo}). 

\begin{proposition}\label{prop:PRI}
Suppose $G$ is a group and $A\seq G$ is finite with $|AA\inv A|\leq c|A|$. Then $|(A\inv A)^3|\leq c^4|AA\inv A|$.
\end{proposition}

\begin{proof}[\textnormal{\textbf{Proof of Theorem \ref{thm:SARtrip}}}]
The result is trivial for $k=1$. So fix $k\geq 2$, $c\geq 1$, and $\epsilon> 0$.  Without loss of generality, assume $\epsilon<  (8(k-1)(30c)^{4k-4})\inv$. Let $\sigma(x)=c^{\nv 1} x^{4k}$ and set $\delta=(\sigma_{k_*,c^4})^{k_*}(\epsilon)$. Then $\delta\geq \Omega_k((c\inv\epsilon)^{O_k(1)})$.

 Now fix a  group $G$ and a finite nonempty $k$-stable set $A\seq G$ with $|AA\inv A|\leq c|A|$. Then $\phi_A(x;y,z)$ is $k_*$-stable and right-invariant. We apply Lemma \ref{lem:tech} to $\phi_A$, with $\sigma$ as above and  $X=AA\inv A$ (so $|X\inv X|\leq c^4|X|$ by Proposition \ref{prop:PRI}). This yields $\eta\in [\delta,\epsilon]$ such that, for any $g\in G$, if $Ag\smd A\seq X$ and $|Ag\smd A|\leq \eta|X|$, then $|Ag\smd A|\leq \sigma(\eta)|X|$. Since $\Stab^A_\eta(A)\seq A\inv A$ and $|X|\leq c|A|$, it follows that $\Stab^A_\eta(A)=\Stab^A_{\eta^{4k}}(A)$. So $H:=\Stab^A_\eta(A)$ is a subgroup of $G$ contained in $A\inv A$. Since $H=\Stab^{A\inv}_\eta(A)$ and $|AA\inv|\leq c|A|$, it follows from Corollary \ref{cor:VC}$(b)$ that $A\seq CH$ for some  $C\seq G$  with $|C|\leq (30c\eta\inv)^{k-1}\leq O_k((c\epsilon\inv)^{O_k(1)})$. Without loss of generality, we can change coset representatives and assume $C\seq A$. 
 
Let $m=|C|$ and fix $g\in C$. Toward a contradiction, suppose we have $|gH\cap A|\geq m\inv\eta|H|$ and $|gH\backslash A|\geq m\inv\eta|H|$. Let $B_1=H\cap g\inv A$ and $B_2=H\backslash g\inv A$. Then, as in the proof of Theorem \ref{thm:SAR},  $B_1$ is $k$-stable and $t$-generic in  $H$ where $t= 8(k-1)(m\eta\inv)^2$. 
So there is $h\in H$ such that 
 \[
 |B_1h\cap B_2|\geq \frac{\eta}{m^2t}|A|>\eta^{4k}|A|,
 \] 
which contradicts $H=\Stab^A_{\eta^{4k}}(A)$. Thus we have that for any $g\in C$, either $|gH\cap A|<m\inv\eta|H|$ or $|gH\backslash A|<m\inv\eta|H|$. Define $D$ to be the set of $g\in C$ such that $|gH\cap A|\geq m\inv\eta|H|$. Since $A\seq CH$, we have $|A\smd DH|<\eta|H|\leq\epsilon|H|$ (using calculations similar to the end of the proof of Theorem \ref{thm:SAR}).
\end{proof}

The proof of Theorem \ref{thm:SARtrip} motivates analogous comments as those made in Remark \ref{rem:mainproof}. For example,
the exponent $O_k(1)$ on $c\epsilon\inv$ in the bound for $|C|$ is at most $k^{\exp^2(2k)}$. By choosing a constant $\epsilon$, we also obtain a strong form of the ``polynomial Freiman-Ruzsa conjecture" for finite $k$-stable sets in arbitrary groups.

\begin{corollary}
Suppose $G$ is a group and $A\seq G$ is a finite nonempty $k$-stable set with $|AA\inv A|\leq c|A|$. Then there is a subgroup $H\leq G$ such that $H\seq A\inv A$ and $A$ is covered by $O_k(c^{k^{\exp^2(2k)}})$ left cosets of $H$.
\end{corollary}

\begin{remark}
As with Theorem \ref{thm:CPT}, the proof of Theorem \ref{thm:MPW} from \cite{MPW} also shows that the group $H$ is a finite Boolean combination of translates of $A$ (of bounded complexity). One can recover this quantitatively using a similar application of Theorem \ref{thm:VCA}. We leave the details as an exercise for the reader. 
\end{remark}

\section{Infinite stable sets}\label{sec:MT}

In the proof of Theorem \ref{thm:SAR}, the main significance of working with finite groups was the behavior of the normalized counting measure. This motivates the question of whether similar results hold in the setting of other measures, e.g., for stable subsets of amenable groups. In this section, we show that this is indeed the case. In fact, this setting is nicely focused due to the result that, in \emph{any group}, one always has a canonical invariant measure on stable sets. To clarify this assertion, let $G$ be a fixed group and let $\cB$ be the Boolean algebra of stable subsets of $G$. It is shown in \cite[Theorem 1.1]{CoLSGT} that there is a unique bi-invariant finitely-additive probability measure on $\cB$, which we denote $\mu$. So, in particular, if $G$ is finite then $\mu$ coincides with the normalized counting measure (restricted to $\cB$). Altogether, the following result generalizes Theorem \ref{thm:SAR} to this setting.

\begin{theorem}\label{thm:SARm}
For any $k$-stable set $A\seq G$, and any $\epsilon>0$, there is a subgroup $H\leq G$ of index $\epsilon^{\nv O_k(1)}$ and a set $D\seq G$, which is a union of left cosets of $H$, such that $\mu(A\smd D)<\epsilon\mu(H)$.
\end{theorem}
\begin{proof}
As usual, we aim for a stronger coset regularity property, namely, we find $H$ so that for all $g\in G$, either $\mu(gH\cap A)<\epsilon\mu(H)^2$ or $\mu(gH\backslash A)<\epsilon\mu(H)^2$. Given $A\in\cB$ and $\epsilon> 0$, let $\Stab^\mu_\epsilon(A)=\{x\in G:\mu(Ax\smd A)\leq\epsilon\}$. In order to mimic the proof of Theorem \ref{thm:SAR} using $\mu$, it suffices to show following properties:
\begin{enumerate}[$(i)$]
\item If $A\in\cB$, $\VC_\ell(A)\leq d$, and $\mu(A)>\epsilon$, then $A$ is $32d\epsilon^{\nv 2}$-generic in $G$.
\item If $A\in\cB$, $\VC_r(A)\leq d$, and $\epsilon>0$, then $\Stab^\mu_\epsilon(A)$ is $(320d)^d\epsilon^{\nv 20d}$-generic in $G$.
\end{enumerate}
We will establish these properties via the following claim, which uses the well-known fact from stability theory that probability  measures associated to stable relations are finitely approximable in the sense of Theorem \ref{thm:VCA} (see, e.g., \cite[Lemma 4.3]{ChStNIP}).
\medskip

\noindent\textit{Claim:} Fix $A\in\cB$ and suppose $0<\epsilon<\mu(A)$. Set $d=\VC_\ell(A)$ (resp., $d=\VC_r(A)$) and $\cS=\{gA:g\in G\}$ (resp., $\cS=\{Ag:g\in G\}$). Then there is $F\seq G$ such that $|F|\leq 32d\epsilon^{\nv 2}$ and $F\cap S\neq\emptyset$ for all $S\in\cS$.

\noindent\textit{Proof:} By \cite[Lemma 4.3]{ChStNIP} applied to $x\in yA$ (resp., $x\in Ay$) there is $\abar\in G^n$ such that, for any $S\in\cS$, $|\mu(S)-\Av_{\abar}(S)|\leq \epsilon/2$. Therefore $\Av_{\abar}(S)>\epsilon/2$ for all $S\in\cS$ by invariance of $\mu$. So the claim follows from Theorem \ref{thm:VC}$(a)$ (view $\Av_{\abar}$ as the normalized counting measure on $\{1,\ldots,n\}$). \claim
\medskip

Conditions $(i)$ and $(ii)$ now follow from the claim as in Corollary \ref{cor:VC}, but we replace Theorem \ref{thm:VC}$(b)$ with \cite[Lemma 4.6]{LovSzeg} (see Remark \ref{rem:VCdim}).
\end{proof}

\begin{remark}
Using \cite[Lemma 4.3]{ChStNIP}, and following the same steps as in Proposition \ref{prop:bool}, one can express the subgroup $H$ in Theorem \ref{thm:SARm} as a Boolean combination (of bounded complexity) of left translates of $A$. Also, as in Theorem \ref{thm:SARn}, one can obtain a version of Theorem \ref{thm:SARm} with a \emph{normal} subgroup of index $\exp^{O_k(1)}(\epsilon\inv)$.
\end{remark}

We conclude with a remark on Corollary \ref{cor:tech}, which was our key use of stability in the proof of Theorem \ref{thm:SAR}. Let $G$  be as above, and suppose $A\seq G$ is stable. Then we have the stabilizers $\Stab^\mu_\epsilon(A)$ for $\epsilon>0$, as in the proof of Theorem \ref{thm:SARm}, and we can analogously define $\Stab^\mu_0(A)$, which is a \emph{subgroup} of $G$. Let $\cB_A$ denote the sub-algebra of $\cB$ generated by all right translates of $A$. By \cite[Theorem 2.3]{CPT} and its proof, $\mu{\upharpoonright}\cB_A$ is the unique right-invariant finitely additive probability measure on $\cB_A$ and, moreover, this measure takes only \emph{finitely many} values (see also \cite[Section 5]{CoLSGT}). So we have the following conclusion, which can also be proved directly via the same inductive argument underlying Lemma \ref{lem:tech}.

\begin{corollary}\label{cor:UP}
If $A\seq G$ is stable then $\Stab^{\mu}_0(A)=\Stab^\mu_\epsilon(A)$ for some $\epsilon>0$.
\end{corollary}

For example, consider the case when $G=\prod_{\cU}G_i$ is an ultraproduct of finite groups and $A=\prod_{\cU}A_i$ is a stable internal subset of $G$. Then $\mu{\upharpoonright}\cB_A$ agrees with the normalized pseudofinite counting measure. So in this case, Corollary \ref{cor:UP} provides a qualitative nonstandard formulation of Corollary \ref{cor:tech} .

\subsection*{Acknowledgements} I would like to thank Caroline Terry and Julia Wolf for helpful discussions, and the anonymous referee for several corrections.

\end{document}